\documentclass{amsart}
\usepackage{amsfonts,amssymb,amscd,amsmath,enumerate,verbatim,calc}
\usepackage[all]{xy}

\newcommand{\B}{\mathcal{B} }
\newcommand{\C}{\mathcal{C} }
\newcommand{\D}{\mathcal{D} }
\newcommand{\T}{\mathcal{T} }
\newcommand{\U}{\mathcal{U} }

\newcommand{\m}{\mathfrak{m} }
\newcommand{\M}{\mathfrak{M} }

\newcommand{\R}{\mathcal{R} }
\newcommand{\F}{\mathcal{F} }
\newcommand{\Sc}{\mathcal{S} }
\newcommand{\Z}{\mathbb{Z} }

\newcommand{\rt}{\rightarrow}

\newcommand{\depth}{\operatorname{depth}}
\newcommand{\height}{\operatorname{height}}

\newcommand{\Proj}{\operatorname{Proj}}

\newcommand{\Ass}{\operatorname{Ass}}

\theoremstyle{plain}

\newtheorem{theorem}{Theorem}[section]

\newtheorem{lemma}[theorem]{Lemma}
\newtheorem{proposition}[theorem]{Proposition}

\theoremstyle{definition}
\newtheorem{definition}[theorem]{Definition}
\newtheorem{s}[theorem]{}
\newtheorem{remark}[theorem]{Remark}

\theoremstyle{remark}

\begin{document}

\title[$R_1$ property]{First Coefficient ideals and $R_1$ property of Rees algebras }
\author{Tony~J.~Puthenpurakal}
\date{\today}
\address{Department of Mathematics, IIT Bombay, Powai, Mumbai 400 076}

\email{tputhen@math.iitb.ac.in}

 \begin{abstract}
Let $(A,\m)$ be an excellent normal local ring of dimension $d \geq 2$ with infinite residue field. Let $I$ be an $\m$-primary ideal. Then the following assertions are equivalent:
\begin{enumerate}[\rm (i)]
  \item The extended Rees algebra $A[It, t^{-1}]$ is $R_1$.
  \item The  Rees algebra $A[It]$ is $R_1$.
  \item $\Proj(A[It])$ is $R_1$.
  \item $(I^n)^* = (I^n)_1$ for all $n \geq 1$.
\end{enumerate}
Here $(I^n)^*$ is the integral closure of $I^n$ and $(I^n)_1$ is the first coefficient ideal of $I^n$.
\end{abstract}
 \maketitle
\section{introduction}
Let $(A.\m)$ be an excellent Noetherian local domain of dimension $d$ and let $I$ be an $\m$-primary ideal. For convenience we assume that the residue field of $A$ is infinite and that $d \geq 2$.
Throughout $\R = A[It, t^{-1}]$ is the extended Rees algebra of $I$ and $\Sc = A[It]$ is the Rees algebra of $I$.  A natural question is when is $\R$ (or $\Sc$) normal, $S_2$ and $R_1$.
If $A$ is also normal then $\R$ and $\Sc$ is normal if and only if $I^n = (I^n)^*$ for all $n \geq 1$ (here $(I^n)^*$ is the integral closure of $I^n$). If $A$ is $S_2$ and quasi-unmixed then by \cite[2.4, 2.5. 2.6]{C},
$\R$ (and $\Sc$) is $S_2$ if and only if $I^n = (I^n)_1$ for all $n \geq 1$ (here $(I^n)_1$ is the first coefficient ideal of $I^n$, see \ref{coeff}). An ideal theoretic characterization of when $\R$ (and $\Sc$) is $R_1$ is not dealt with in the literature. In this short article we address this issue.

We prove:
\begin{theorem}\label{main}
Let $(A,\m)$ be an excellent normal local ring of dimension $d \geq 2$ with infinite residue field. Let $I$ be an $\m$-primary ideal. Then the following assertions are equivalent:
\begin{enumerate}[\rm (i)]
  \item The extended Rees algebra $\R$ is $R_1$.
  \item The  Rees algebra $\Sc$ is $R_1$.
  \item $\Proj(\Sc)$ is $R_1$.
  \item $(I^n)^* = (I^n)_1$ for all $n \geq 1$.
\end{enumerate}
\end{theorem}

The main contribution of this paper is to guess the result. The proofs are not particularly hard.
Here is an overview of the contents of this paper. In section two we discuss a few preliminaries that we need. In section three we prove two technical results that we need. Finally in section four we prove Theorem \ref{main}.
\section{Preliminaries}
In this section we discuss a few preliminaries that we need. In this paper all rings considered are Noetherian and all modules considered (unless otherwise stated) is finitely generated.
For all undefined terms see \cite{BH}.

\s Let $A$ be a  ring and let $M$ be an  $A$-module. We say that $M$ satisfies Serre’s $S_2$ property if for every prime ideal $P$ of $A$
\[
 \depth M_P \geq  \inf\{2,\dim M_P\}.
 \]
 We say that the ring $A$ satisfies $S_2$ if it satisfies $S_2$ as an $A$-module.

\begin{definition}
Let $A$ be a  domain with quotient field $Q(A)$. We say that a domain $B$ is an $S_2$-ification of A if:
\begin{enumerate}
\item
$A \subseteq B \subseteq Q(A)$ and $B$ is module-finite over $A$,
\item
$B$ is $S_2$ as an $A$-module, and
\item
 for all $b$ in $B \setminus A$, $\height D(b) \geq 2$ where $D(b) = \{ a \in A \mid ab \in A \}$.
\end{enumerate}
\end{definition}
\begin{remark}
In general, the $S_2$-ification of a domain might not exist, but if there is one, then it must be unique; see \cite[2.4]{HH}.
\end{remark}

\s \emph{Coefficient ideals:} Let $(A,\m)$ be a local ring and let $I$ be an $\m$-primary ideal. For sufficiently large values of $n$, $\ell(A/I^{n+1})$ is a polynomial $P_I(n)$ in $n$ of degree $d = \dim A$, the Hilbert polynomial of $I$. We write this polynomial in terms of binomial coefficients:
\[
P_I(n) = e_0(I)\binom{n+d}{d} - e_1(I)\binom{n+d -1 }{d -1} + \cdots + (-1)^de_d(I).
\]
The coefficients $e_i(I)$ are integers and we call them the Hilbert coefficients of $I$. In \cite[Theorem 1]{S} the following theorem is proved.
\begin{theorem}\label{coeff}
Let $(A,\m)$ be a quasi-unmixed local ring of dimension $d > 0$ with infinite residue field and $I$ an $\m$-primary ideal. Then for each integer $k$ in $\{0,1,\cdots,d \}$ there exists a unique largest ideal $I_{k}$ containing $I$ such that $e_i(I) = e_i(I_{k})$ for $i = 0,1,\cdots,k$. The ideal $I_{k}$ is called the $k$-th coefficient ideal of $I$.
\end{theorem}
$I_{0}$ is the integral closure of $I$ and if $I$ contains a regular element, then $I_{d}$ is the Ratliff-Rush closure of $I$.

\s \emph{$S_2$-fication of the extended Rees algebra:} We will need the following result
\begin{theorem}\label{cc}[\cite[2.4, 2.5. 2.6]{C} ]
Let $(A, \m)$ be a quasi-unmixed, analytically un-ramified local domain with infinite residue field and  dimension $d \geq 2$. Assume $A$ satisfies $S_2$ as a ring. Let $\R = A[It, t^{-1}]$ be the extended Rees ring of $A$. Then the $S_2$-fication of $\R$ is
$\bigoplus_{n \in \Z} I_n$ where $I_n = A$ for $n \leq 0$ and $I_n = (I^n)_1$ for $n \geq 1$.
\end{theorem}

\s\label{ver} Let $B = \bigoplus_{n \geq 0} B_n$ be a standard graded algebra over $B_0$. We assume $B_0$ is local. Let $M = \bigoplus_{n \geq 1}M_n$ be a finite $B$-module with $\ell(M_n) $ finite for all $n \geq 1$. Set $M_{\geq m} = \bigoplus_{n \geq m}M_n$. We note that for all $m \gg 0$ (say $m \geq r$) the $B$-module $M_{\geq m}$ is generated by $M_m$. Let $B^{(c)}$ be the Vernese subring of $B$ (for $c \geq 1$). Let $M^{(c)} = \bigoplus_{n \geq 1}M_{cn}$ be the Veronese submodule of $M$. The following result is
definitely known. We give a proof due to lack of a reference.
\begin{lemma}\label{v-lem}(with setup as in \ref{ver}) Let $c \geq r$. Then
$$\dim M \leq \max\{\dim M^{(c)}, 0 \}.$$
\end{lemma}
\begin{proof}
  We may assume $\dim M  > 0$. Set $N = M_{\geq c}$. Note $\dim N = \dim M$. After going mod annhilator of $N$ we may assume that $N$ is a faithful $B$-module. It is also generated in degree $c$.
  Let $N = (n_1, \ldots, n_l)$ where $n_i \in N_c$. Note we have an exact sequence
  \[
  0 \rt B \xrightarrow{\phi} N^l(+c)
  \]
  where $\phi(b) = (bn_1, \ldots, bn_l)$. Taking Veronese we have an exact sequence
  \[
  0 \rt B^{(c)} \xrightarrow{\phi^{<c>}} (N^{<c>})^l(+1)
  \]
  So it follows that $\dim N^{<c>} = \dim B^{<c>}$. Note $\dim B = \dim B^{<c>}$ and as $N$ is a faithful $B$-module we have $\dim N = \dim B$.
  Thus $\dim N^{<c>} = \dim M$. As $M_0 = 0$ note that $M^{<c>} = N^{<c>}$.
\end{proof}
We will also need the following result.
\begin{lemma}\label{first-dim}
Let $(A,\m)$ be a quasi-unmixed local ring of dimension $d > 0$ with infinite residue field and $I$ an $\m$-primary ideal. Set $J = I_1$ the first coefficient ideal of $I$. Set $\Sc = A[It]$ and
$\T = A[Jt]$. Consider the finite $\Sc$-module $\T/\Sc$. Then $\dim \T/\Sc \leq d - 1$.
\end{lemma}
\begin{proof}
As $e_0(I) = e_0(J) $ and $e_1(I) = e_1(J)$ the function $n \mapsto \ell(J^n/I^n)$ is of polynomial type of degree $\leq d-2$. The result follows.
\end{proof}
\section{Main technical results}
In this section we prove two technical results that we need. The first technical result is :
\begin{theorem}
\label{tech} Let $(A,\m)$ be a quasi-unmixed, analytically unramified, local, $S_2$-domain of dimension $d \geq 2$. Assume the residue field of $A$ is infinite. Let $I$ be an $\m$-primary ideal. Let  $\C = \bigoplus_{n \geq 1}(I^n)^*/(I^n)_1$ considered as an $\Sc = A[It]$-module.
Then either $\dim \C = d$ or $\C = 0$.
\end{theorem}
\s\label{techs}
 (with hypotheses as in \ref{tech}) Set $\T = \bigoplus_{n \in \Z}I_n$ where $I_n = A$ for $n \leq 0$ and $I_n = (I^n)_1$. Then $\T$ is the $S_2$-fication of the extended Rees algebra $\R = A[It, t^{-1}]$, \cite[2.4, 2.5. 2.6]{C}. In particular $\F = \{ (I^n)_1\}_{n \geq 0}$ is a multiplicative  $I$-stable filtration of $A$ (we set $(I^0)_1 = A_1 = A$). Let $\U = \bigoplus_{n \geq 0}(I^n)_1$.

Set $L = \bigoplus_{n \geq 1}A/(I^n)_1$. We have a following exact sequence
\[
0 \rt \U \rt A[t] \rt L \rt 0.
\]
This gives $L$ the structure of an $\U$-module. We note that $L$ is not finitely generated as an $U$-module.

Let $G = G(\F) = \bigoplus_{n \geq 0}(I^n)_1/(I^{n +1})_1$ be the associated graded ring of the filtration $\F$. The following result was proved for the $I$-adic filtration in \cite[5.6]{P}. Its extension to arbitrary multiplicative $I$-stable filtration's is routine.
\begin{lemma}
\label{ass} We have
$$ \Ass_\U G = \Ass_\U L.$$
\end{lemma}
We now give
\begin{proof}[Proof of Theorem \ref{tech}]
We have an inclusion of rings $\Sc \subseteq \U \subseteq \Sc^*$. As $\Sc^*$ is a finite $\Sc$-module it follows that $\Sc^*$ is a finite $\U$-module.
So $\C = \Sc^*/\U$ is a finite $\U$-module.

 We note that $\T$ is a catenary $S_2$-domain.
It follows that $G = \T/(t^{-1})$ is $S_1$ and equi-dimensional, see \cite[Lemma 2, p.\ 250]{Ma} for the local case (same proof works in the $*$-local case).

We note that $\C$ is an $\U$-submodue of $L$. Suppose $\C \neq 0$.
Let $P $ be a minimal prime of $C$. Note
$P \in \Ass_\U L = \Ass_\U G$ (see \ref{ass}).  As $G$ is $S_1$ it is in particular unmixed. Also it is equi-dimensional. It follows that $\dim \U/P = d$. The result follows.
\end{proof}
The second technical result that we need is:
\begin{proposition}
\label{dim-first}(with hypotheses as in \ref{tech}) Let $\D = \bigoplus_{n \geq 1}(I^n)_1/I^n$. Then $\D$ is a finite $\Sc$-module of dimension $\leq d -1$.
\end{proposition}
\begin{proof}
  We have an inclusion of rings $\Sc \subseteq \U \subseteq \Sc^*$. As $\Sc^*$ is a finite $\Sc$-module it follows that $\U$ is a finite $\Sc$-module. So $\D = \U/\Sc$ is a finite $\Sc$-module.

  We may assume that for $m \geq r$ the $\Sc$-module $\D_{\geq m}$ is generated in degree $m$. Note some Veronese of $\U$ (say $l$) is standard graded.   Set $c = rl$
  and $J = (I^{c})_1$. Then $D^{<c>} = A[Jt]/A[I^{c}t]$. By \ref{first-dim} we get $\dim D^{<c>} \leq d -1$. The result follows from  \ref{v-lem}.
\end{proof}
\section{Proof of Theorem \ref{main}}
In this section $(A,\m)$ is an excellent normal local ring of dimension $d \geq 2$ and $I$ is an $\m$-primary ideal. Throughout $\R = A[It, t^{-1}]$ is the extended Rees algebra of $I$ and $\Sc = A[It]$ is the Rees algebra of $I$. Set $\R^* = \bigoplus_{n \in \Z} (I^n)^*$ and $\Sc^* = \bigoplus_{n \geq 0}(I^n)^*$. Note $\R^*$ is a finite $\R$-module and they have the same field of fractions.
Similarly $\Sc^*$ is a finite $\Sc$-module and they have the same field of fractions. Furthermore both $\R^*$ and $\Sc^*$ are normal domains.
\s \label{b} Set $\B = \bigoplus_{n \geq 1}(I^n)^*/I^n.$ Then the following short exact sequences give $\B$ a structure both as an $\R$-module and as a $\Sc$-module.
\[
0 \rt \R \rt \R^* \rt \B \rt 0.
\]
\[
0 \rt \Sc \rt \Sc^* \rt \B \rt 0.
\]
As $\ell(\B_n)$ is finite for all $n  \geq 1$ and so  $\dim \B$ is determined by its Hilbert function. It follows that $\B$ has the same dimension both as an $\R$-module and as an $\Sc$-module.

The following result covers Theorem \ref{main}.
\begin{theorem}
  \label{main-body} Let $(A,\m)$ be an excellent normal local ring of dimension $d \geq 1$ and let $I$ be an $\m$-primary ideal. The following assertions are equivalent.
  \begin{enumerate}[\rm (i)]
    \item $\R$ is $R_1$.
    \item $\dim \B \leq d-1$ as an $\R$-module.
    \item $\Sc$ is $R_1$.
    \item $\dim \B \leq d-1$ as an $\Sc$-module.
    \item $\Proj \Sc$ is $R_1$.
    \item $(I^n)^* = (I^n)_1$ for all $n \geq 1$.
  \end{enumerate}
\end{theorem}
\begin{proof}
$\rm{(i)} \implies \rm{(ii)}$: Let $P$ be a height one prime of $\R$.
We have an exact sequence $0 \rt \R_P \rt (\R^*)_P \rt \B_P \rt 0$.
By our hypothesis $\R_P$ is a DVR. We also have that $(R^*)_P$ is a finite $\R_P$-module and it has same field of fractions as that of $\R_P$. It follows that $\R_P = (\R^*)_P$. So $\B_P = 0$.
The result follows.

$\rm{(ii)} \implies \rm{(i)}$: Let $P$ be a prime of height one in $\R$. Then $\B_P = 0$. It follows that  $\R_P = (\R^*)_P$. So $\R_P$ is a DVR. It follows that $\R$ is $R_1$.

$\rm{(ii)} \Leftrightarrow \rm{(iv)}$: This follows from \ref{b}.

$\rm{(iii)} \Leftrightarrow \rm{(iv)}$: This is similar to $\rm{(i)} \Leftrightarrow \rm{(ii)}$.

$\rm{(iii)} \implies \rm{(v)}$: This is trivial.

$\rm{(v)} \implies \rm{(iv)}$: We may assume $\dim \B > 0$, otherwise there is nothing to prove. Let $P$ be a minimal prime of $\B$. By assumption $P \neq \M$, the maximal homogeneous ideal of $\Sc$.  We claim that $P \nsupseteq \Sc_+$. Suppose if possible $P \supseteq \Sc_+$. As $\B$ is a finite $\Sc$-module and as $\ell(\B_n)$ is finite for all $n$ there exists $r$ such that $\m^r \B = 0$. So $P $ contains
$(\m^r  \oplus \Sc_+)$. It follows that $P = \M$, a contradiction. Rest of the argument follows as in $\rm{(i)} \implies \rm{(ii)}$.

$\rm{(iv)} \implies \rm{(vi)}$: Let  $\C = \bigoplus_{n \geq 1}(I^n)^*/(I^n)_1$ considered as an $\Sc$-module. Then $\C$ is a quotient of $\B$. So $\dim \C \leq d-1$. By \ref{tech}
we get $\C = 0$. So $(I^n)^* = (I^n)_1$ for all $n \geq 1$.

$\rm{(vi)} \implies \rm{(iv)}$: By our assumption $\B = \bigoplus_{n \geq 1}(I^n)_1/I^n.$
By \ref{dim-first}, it follows that $\dim \B \leq d -1$.
\end{proof}


\begin{thebibliography} {99}








\bibitem {BH}  W. Bruns and J. Herzog,
\emph{Cohen-Macaulay Rings}, revised edition,
Cambridge Studies in Advanced Mathematics, 39.
Cambridge University Press, 1998.

\bibitem{C}
C.~Ciuperc\u a,
\emph{First coefficient ideals and the  $S_2$ -ification of a Rees algebra},
J. Algebra 242 (2001), no. 2, 782–-794.

\bibitem{HH}
M.~Hochster and C.~Huneke,
\emph{Indecomposable canonical modules and connectedness},
in "Commutative Algebra: Syzygies, Multiplicities, and Birational Algebra (South Hadley, MA, 1992),"
Contemporary Mathematics, Vol. 159, pp. 197-–208, Am. Math. Soc., Providence, 1994.


\bibitem{Ma}
H.~Matsumura, \emph{Commutative ring theory}, Cambridge Studies in Advanced
  Mathematics, vol.~8, Cambridge University Press, Cambridge, 1986.


\bibitem{P}
T.~J.~Puthenpurakal,
 \emph{A sub-functor for Ext and Cohen-Macaulay associated graded modules with bounded multiplicity},
Trans. Amer. Math. Soc. 373 (2020), no. 4, 2567--2589.

\bibitem{S}
K.~Shah,
\emph{Coefficient ideals},
 Trans. Amer. Math. Soc. 327, No. 1 (1991), 373–-384.



\end{thebibliography}
\end{document}